%% file: barriers3.tex
\documentclass{llncs}
\title{A Canonical Characterization of the Family of Barriers in General Graphs}

\author{Nanao Kita}
\institute{Keio University, Yokohama, Japan\\
           \email{kita@a2.keio.jp}}

\usepackage{amsmath, amssymb}
\usepackage{algorithmic}
\usepackage{algorithm}
\usepackage{cite}
\usepackage[stable]{footmisc}
\usepackage[dvipdfmx]{graphicx}

\usepackage{setspace}
\usepackage{color}
\usepackage{comment}

\newcommand{\matchablesp}{factorizable~}

\newcommand{\zero}{balanced~}

\newcommand{\exposed}{exposed}

\newcommand{\Vg}{V(G)}

\newcommand{\dm}{\preceq}
\newcommand{\pardm}[1]{\dm_{#1}}
\newcommand{\yield}{\triangleleft}
\newcommand{\gpart}[1]{\mathcal{P}(#1)}
\newcommand{\pargpart}[2]{\mathcal{P}_{#1}(#2)}
\newcommand{\gsim}[1]{\sim_{#1}}
\newcommand{\hxg}[2]{H_{#2}(#1)}
\newcommand{\up}[1]{\mathcal{U}(#1)}
\newcommand{\upstar}[1]{\mathcal{U}^{*}(#1)}
\newcommand{\parup}[2]{\mathcal{U}_{#1}(#2)}
\newcommand{\parupstar}[2]{\mathcal{U}^{*}_{#1}(#2)}
\newcommand{\vup}[1]{U(#1)}
\newcommand{\vupstar}[1]{U^*(#1)}
\newcommand{\vparup}[2]{U_{#1}(#2)}
\newcommand{\vparupstar}[2]{U^*_{#1}(#2)}

\newcommand{\vcoup}[1]{{}^cU(#1)}

\newcommand{\what}{\widehat}
\renewcommand{\Gamma}{N}
\renewcommand{\Delta}{\triangle}

\newcommand{\astar}[2]{X}

\newcommand{\ignore}[1]{}
\newcommand{\sfig}[1]{{} }

\spnewtheorem{cclaim}{Claim}{\bfseries}{\rmfamily}
\spnewtheorem*{ac}{Acknowledgement}{\bfseries}{\rmfamily}
\spnewtheorem*{proofof}{Proof of}{\itfamily}{\rmfamily}
\spnewtheorem*{maintheorem}{Main Theorem}{\bfseries}{\rmfamily}

\begin{document}
\maketitle 
\pagestyle{plain}

\begin{abstract}
\input{barriers_abstract_barrier}

\end{abstract}

\section{Introduction}\label{sec:intro}

\input{cocoa2013_sec_intro_barrier}

\section{Preliminaries}\label{sec:pre}
\subsection{Definitions and Some Preliminary Facts}
\input{barriers_sec_pre_notations}
\input{barriers_sec_pre_propositions}
\input{barriers_prop_intersection}
\input{barriers_prop_intersection_path}

\input{barriers_prop_two-element-extreme}

\input{barriers_prop_circ2elem}

\subsection{The Dulmage-Mendelsohn Decomposition} 
\input{barriers_sec_pre_dm}

\subsection{Our Aim%
\protect\footnote{For more details on the statements in this section, 
see Appendix.}}\label{sec:aim}
\input{cocoa2013_sec_motivation} 
\subsection{The Generalized Cathedral Structure}\label{sec:canonical}
\input{barriers_sec_canonical_short}
\section{A Generalization of Lov\'asz's Canonical Partition}%
\label{sec:barrier}
\subsection{Our Main Result}
\input{cocoa2013_sec_maintheorem}
\subsection{Barriers vs. Alternating Paths}\label{sec:projection}
In this subsection we introduce some lemmas on  
the reachability of alternating paths regarding odd-maximal barriers. 
Given an odd-maximal barrier $X$ of a factorizable graph $G$, 
we generate a bipartite graph,  
thus canonically decompose $X\cup D_X$ and 
state the reachability using the DM-decomposition as a language.  
This technique of generating a bipartite graph 
has been known \cite{lp1986, frank1993} and 
essences of ideas are found there. 
However, we first reveal it thoroughly 
to obtain Proposition~\ref{prop:pathprojection} and Theorem~\ref{thm:path2dm}. 
\input{barriers_prop-proof_extreme2path_x2d_closed}

\input{cocoa2013_prop_x2xd}
\input{cocoa2013_proof_prop_x2xd}
\input{cocoa2013_def_hxg}
The next proposition is easily seen 
by Propositions~\ref{prop:barrier} and \ref{prop:odd-maximal} 
and enables us to discuss Proposition~\ref{prop:pathprojection} and so on. 
%
\input{cocoa2013_prop_projection}
\input{cocoa2013_proof_prop_projection}
The next proposition shows that the reachabilities 
of alternating paths are equivalent between $G$ and $\hxg{X}{G}$, 
which, with Proposition~\ref{prop:dm}, 
derives Theorem~\ref{thm:path2dm} immediately. 
\input{cocoa2013_prop_pathprojection}
\input{cocoa2013_proof_prop_pathprojection}
%
\input{cocoa2013_comment_color}
\input{cocoa2013_def_expansion}
The next proposition is a basic observation on expansions.  
%
\input{cocoa2013_prop_dm2partition}
\input{cocoa2013_proof_prop_dm2partition}
\input{cocoa2013_thm_path2dm}
\input{cocoa2013_proof_thm_path2dm}
%
\input{cocoa2013_lem_bip-traverse}
\input{cocoa2013_proof_lem_bip-traverse}
The following lemma is obtained by 
Propositions~\ref{prop:pathprojection} and \ref{prop:dm2partition}, 
and Theorem~\ref{thm:path2dm}. 
%
\input{cocoa2013_lem_traverse}
\input{cocoa2013_proof_lem_traverse}
\subsection{Canonical Structures of  Odd-maximal Barriers}\label{sec:barrier2up}
In this subsection we examine the reachability of alternating paths 
regarding the cathedral structure and derive the main theorem. 
%
%
%
\input{cocoa2013_prop_base}
\begin{proof}
This is a mere restatement of Proposition~\ref{prop:block}. 
\qed
\end{proof}
The next lemma is obtained by 
Proposition~\ref{prop:block} and Proposition~\ref{prop:path2root}. 
\input{cocoa2013_lem_up-path}
\input{cocoa2013_proof_lem_up-path}
Immediately by Theorem~\ref{thm:generalizedcanonicalpartition},  
we can see the next proposition:  
\input{cocoa2013_prop_deletable2path}
\input{cocoa2013_proof_prop_maximal}
The next one is by Proposition~\ref{prop:ear} and Lemma~\ref{lem:up-path}. 
\input{cocoa2013_lem_nosaturate}
\input{cocoa2013_proof_lem_nosaturate}
The next one,  Lemma~\ref{lem:distinctivebase},  is rather easy to see 
by Proposition~\ref{prop:ear}, 
and combining it with Lemma~\ref{lem:up-path} 
we can obtain Lemma~\ref{lem:path2up}. 
\input{cocoa2013_lem_distinctivebase}
\input{cocoa2013_proof_lem_distinctivebase}

\input{cocoa2013_lem_path2up}
\input{cocoa2013_proof_lem_path2up}
Lemma~\ref{lem:path2up} immediately yields
 the following: Lemma~\ref{lem:path2up_restated}. 
\input{cocoa2013_lem_path2up_restated}
\input{cocoa2013_thm_char2elem}
\input{cocoa2013_proof_thm_char2elem}
Now we are ready to prove the main theorem, 
combining up the results in this section. 
\input{cocoa2013_thm_barrier2up}
\input{cocoa2013_proof_thm_barrier2up}

\input{cocoa2013_remark_special}
\input{cocoa2013_remark_atom}
\input{cocoa2013_remark_nonfactorizable}

\section{A Slightly More Efficient
 Algorithm to Compute the Cathedral Structure}\label{sec:algeff}
\input{isaac2012_1_sec_algeff}
\bibliographystyle{splncs}
\bibliography{barriers.bib}

\newpage
\section*{Appendix: Backgrounds on Odd-maximal Barriers}
\input{cocoa2013_appendix_why}

\end{document}

%% file: barriers_abstract_barrier.tex
Given a graph, a barrier is  a set of vertices determined by 
the Berge formula---the min-max theorem characterizing the size of 
maximum matchings. 
The notion of barriers plays important roles 
in numerous contexts of matching theory, 
since barriers essentially coincides with dual optimal solutions 
of the maximum matching problem. 
%
%
%
In a special class of graphs called the elementary graphs,  
the family of maximal barriers forms a partition of the vertices; 
this partition was found by Lov\'asz and is called the canonical partition. 
The canonical partition has produced many fundamental results in matching theory, 
such as the two ear theorem.  
However, in non-elementary graphs, 
the family of maximal barriers never forms a partition,  
and there has not been the canonical partition for general graphs. 
In this paper, 
using our previous work,  
we give a canonical description  of structures of the odd-maximal barriers%
---a class of barriers including the maximal barriers---for general graphs; 
we also reveal structures of odd components associated with odd-maximal barriers. 
This result of us can be regarded as a generalization of Lov\'asz's canonical partition.

%% file: cocoa2013_sec_intro_barrier.tex
A {\em  matching} of a graph $G$ is a 
set of edges no two of which  have  common vertices.
A matching of cardinality $|V(G)|/2$ (resp. $|V(G)|/2 - 1$) is called 
a {\em  perfect matching} (resp. a {\em  near-perfect matching}).
We call a graph \textit{factorizable} if it has at least one perfect matching. 
Now let $G$ be a factorizable graph. 
An edge $e\in E(G)$  is called \textit{allowed} if 
there is a perfect matching containing $e$.
Let $\what{M}$ be the union of all the allowed edges of $G$. 
For each connected component $C$ of the subgraph of $G$ determined by $\what{M}$, 
we call the subgraph of $G$ induced by $V(C)$ as 
{\em factor-connected component} or {\em factor-component} for short. 
The set of all the factor-components of $G$ is denoted by $\mathcal{G}(G)$. 
Therefore, a factorizable graph is composed of 
factor-components and some edges joining between 
different factor-components. 
A factorizable graph with exactly one factor-component 
is called {\em elementary}.

Matching theory is of central importance in graph theory and combinatorial optimization,  
with numerous practical applications~\cite{cc2005}. 
In matching theory, the notion of barriers plays significant roles. 
Given a graph, we call a connected component of it with an odd (resp. even)  number of vertices 
\textit{odd component} (resp. {\em even component}). 
Given $X\subseteq V(G)$ of a graph $G$, 
we denote as $q_{G}(X)$ the number of odd components 
that the graph resulting from deleting $X$ from $G$ has;  
we denote the cardinality of a maximum matching of $G$ as $\nu(G)$. 
There is a min-max theorem called the {\em  Berge formula}~\cite{lp1986} that 
for any graph $G$,
$|V(G)| - 2\nu(G) = \mathrm{max}\{ q_{G}(X) - |X| : X\subseteq V(G) \}$.
%
A set of vertices that attains the maximum 
in the right side of the equation 
is called  a {\em  barrier}.  
Roughly speaking,  barriers essentially coincide with dual optimal  solutions of 
the maximum matching problem, 
and decompose graphs so that 
one can see the structures of maximum matchings. 
However, compared to numerous results on maximum matchings, 
``much less is known about barriers~\cite{lp1986}''.

There is a structure of elementary graphs called the {\em canonical partition}; 
Kotzig first introduced it as the equivalence classes of a certain equivalence relation, 
and later Lov\'asz reformulated it from the point of view of barriers, 
stating that 
the family of maximal barriers forms a partition of the vertices in elementary graphs. 
This reformulation by Lov\'asz has produced many fundamental properties in matching theory 
such as the \textit{two ear theorem}~\cite{lp1986, cc2005}, 
and the \textit{brick decomposition}
 or the \textit{tight cut decomposition}, 
 and underlies polyhedral studies of matching theory; 
 see the survey article~\cite{clm2003}.  

However, in non-elementary graphs, 
the family of maximal barriers never forms a partition of the vertices,  
and there has not been known 
the counterpart structure 
of Lov\'asz's canonical partition for general graphs. 
In this paper, therefore, 
we reveal canonical structures of maximal barriers 
and obtain a generalization of 
Lov\'asz's canonical partition for general graphs; 
here, our previous work on 
canonical structures of general factorizable graphs~\cite{kita2012a, kita2012b},  
the {\em generalized cathedral structure} (see Section~\ref{sec:canonical}), 
serves as a language to describe barriers. 
(Actually, we work on a wider notion called {\em odd-maximal barriers}; 
see Section~\ref{sec:aim}. ) 
In \cite{kita2012a, kita2012b}, 
we defined an equivalence relation and introduced 
a generalization of the canonical partition based on Kotzig's formulation: 
the {\em generalized canonical partition}. 
In this paper,  we show that 
it can be also regarded as 
a generalization based on Lov\'asz's formulation, stating that  
the family of equivalence classes of the generalized canonical partition 
are ``atoms'' that constitute (odd-)maximal barriers in general graphs 
(which shall be introduced in Section~\ref{sec:barrier}). 
We also reveal the structure of odd components associated with (odd-)maximal barriers. 
 
Because the canonical partition 
and the notion of barriers are important, 
we are sure that our result will produce many applications in matching theory. 
There has been known a close relationship between 
algorithms in matching theory, 
barriers, and canonical structure theorems~\cite{lp1986, cc2005}; 
therefore, 
 our result will have algorithmic applications. 
Lov\'asz's canonical partition 
has been the foundation in the study of polyhedral aspects of matchings; 
therefore, 
 our results will make a contribution to this field. 
So far we have already obtained some consequences~\cite{kita2012c} 
on  the {\em optimal ear-decomposition}~\cite{frank1993}. 

%% file: barriers_sec_pre_notations.tex

In this paper we mostly observe those given by Schrijver~\cite{schrijver2003} for 
standard definitions and notations. 
We list here those additional or somewhat non-standard. 

Hereafter for a while let $G$ be a graph.  
We denote the vertex set and the edge set of $G$ 
as $V(G)$ and $E(G)$. 
Now for a while let $X\subseteq V(G)$. 
We define the {\em contraction} of $G$ by $X$ as 
the graph obtained by contracting $X$ into one vertex, 
and denote it as $G/X$. 
For simplicity, 
we identify vertices, edges, subgraphs of $G/X$ 
with those of $G$ naturally corresponding to them. 
The subgraph of $G$ induced by $X$ is denoted by $G[X]$. 
We denote by $G-X$ the graph $G[\Vg\setminus X]$.  

The set of edges that has one end vertex  in $X\subseteq V(G)$ and 
the other end in $Y\subseteq V(G)$ is denoted as $E_G[X, Y]$.
We denote $E_G[X, V(G)\setminus X]$ as $\delta_G(X)$.
We define the \textit{set of neighbors}  of $X$ as
the set of vertices in $V(G)\setminus X$ that 
are adjacent to vertices in $X$, and denote as $\Gamma_G(X)$.
We sometimes denote $E_G[X, Y]$, $\delta_G(X)$, $\Gamma_G(X)$ as just  
$E[X,Y]$, $\delta(X)$, $\Gamma(X)$
if their meanings are apparent from contexts. 
Given two graphs $G_1$ and $G_2$, 
we denote by $G_1 +  G_2$ 
the graph $\left( V\left( G_1 \right) \cup V\left( G_2\right),
 E\left( G_1\right)\cup E\left( G_2\right) \right)$.  
Let $\hat{G}$ be a supergraph of $G$.
For $e = uv\in E(\hat{G})$,
$G-e$ means the graph $(V(G), E(G)\setminus \{e\})$. 
For a set of edges $F = \{e_i\}_{i=1}^k$,
$G -F$ means the graph $G - e_1 -\cdots - e_k$. 
%
%

In many contexts, we often regard a subgraph $H$ of $G$ as 
a vertex set $V(H)$.  
For example, $G/H$ means $G/V(H)$. 
We treat paths and circuits as graphs.
For a path $P$ and $x, y \in V(P)$,
$xPy$ means the subpath of $P$ whose end vertices are $x$ and $y$. 

We say a matching $M$ of $G$ {\em exposes} $v\in V(G)$ 
if $\delta(v)\cap M = \emptyset $, 
otherwise say it {\em covers } $v$. 
For  a matching $M$ of $G$ and $u\in V(G)$ covered by $M$,
$u'$ denotes the vertex to which $u$ is matched by $M$.
For $X\subseteq V(G)$, 
$M_X$ denotes $M\cap E(G[X])$. 

Hereafter for a while let $M$ be a matching of $G$.
For a subgraph $Q$ of $G$, which is a path or circuit,
we call $Q$  $M$-{\em  alternating}
if $E(Q)\setminus M$ is a matching of $Q$.
Let $P$ be an $M$-alternating path of $G$ with end vertices $u$ and $v$.
If $P$ has an even number of edges 
and $M\cap E(P)$ is a near-perfect matching of $P$ exposing only $v$, 
we call it an $M$-{\em  \zero path} from $u$ to $v$.
We regard a trivial path, that is, a path composed of 
one vertex and no edges  as an $M$-\zero path.
If $P$ has an odd number of edges 
and $M\cap E(P)$ (resp. $E(P)\setminus M$) is a perfect matching of $P$,
we call it $M$-{\em  saturated} (resp. $M$-{\em  \exposed}).
%
%

 We say a path $P$ of $G$  is an {\em  ear relative to} $X\subseteq V(G)$ 
if both end vertices of $P$ are in $X$ while internal vertices are not.
We also call a circuit an ear relative to $X$ if exactly one vertex of it is in $X$. 
For simplicity, we call the vertices of $V(P)\cap X$ 
{\em end vertices} of $P$, even if $P$ is a circuit.
For an ear $P$ of $G$ relative to $X$,
we call it an $M$-{\em  ear}
if  $P-X$ is an $M$-saturated path. 
Given an ear $P$ and $Y \subseteq V(G)$, 
we say $P$ is {\em through } $Y$ if 
$P$ has some internal vertices in $Y$.

%
%
%

%% file: barriers_sec_pre_propositions.tex
%
Hereafter in this section we present some basic properties 
used explicitly or implicitly throughout this paper. 
These are easy to see and  
the succeeding two propositions are well-known and might be folklores. 
A graph is called {\em  factor-critical} if 
each deletion of an arbitrary vertex results in an empty graph or a factorizable graph. 
\begin{proposition}[folklore]\label{prop:path2root}
Let $M$ be a near-perfect matching of a graph $G$ that exposes $v\in \Vg$. 
Then, $G$ is factor-critical if and only if for any $u\in \Vg$ there exists 
an $M$-\zero path from $u$ to $v$.
\end{proposition}
Given a graph $G$ and $X\subseteq V(G)$,
we denote the vertices  contained in 
the odd components  of $G-X$ as $D_X$, and  
$V(G)\setminus X \setminus D_X$ as $C_X$. 
The next proposition can be easily observed by the Berge formula. 
\begin{proposition}[folklore]\label{prop:barrier}
Let $G$ be a factorizable graph, and $X\subseteq V(G)$ be a barrier of $G$.
Then for any perfect matching $M$ of $G$,
\renewcommand{\labelenumi}{\theenumi}
\renewcommand{\labelenumi}{{\rm (\theenumi)}}
\renewcommand{\theenumi}{\roman{enumi}}
\begin{enumerate}
\item each vertex of $X$ is  matched to a vertex of $D_X$,
\item for each component $K$ of $G[D_X]$, 
$M_K$ is a near-perfect matching of $K$, accordingly 
$|\delta(K)\cap M |= 1$,  
\item $M$ contains a perfect matching of $G[C_X]$, and 
\item no edge in $E[X, C_X]$ nor $E(G[X])$ is allowed.
\end{enumerate}
\end{proposition}
Now let $G$ be a factorizable graph. 
We say $X\subseteq V(G)$ is \textit{separating} 
if any $H\in\mathcal{G}(G)$ satisfies 
$V(H)\subseteq X$ or $V(H)\cap X = \emptyset$. 
The next one is easy to see by the definitions. 
\begin{proposition}\label{prop:separating}
Let $G$ be a factorizable graph, and let $X\subseteq V(G)$. 
Then, the following four properties are equivalent:  
\renewcommand{\labelenumi}{\theenumi}
\renewcommand{\labelenumi}{{\rm \theenumi}}
\renewcommand{\theenumi}{(\roman{enumi})}
\begin{enumerate}
\item The set $X$ is separating.  
\item The set $X$ is an empty set, or there exist $H_1, \ldots, H_k \in \mathcal{G}(G)$ 
such that $X = V(H_1)\dot{\cup}\cdots\dot{\cup}V(H_k)$. 
\item For any perfect matching $M$ of $G$, $\delta(X)\cap M = \emptyset$. 
\item For any perfect matching $M$ of $G$, 
$M_X$ forms a perfect matching of $G[X]$.
\end{enumerate}
\end{proposition}

%% file: barriers_prop_intersection.tex
\begin{proposition}\label{prop:intersection}
Let $G$ be a factorizable graph, and $M$ be a perfect matching of $G$. 
Let $X\subseteq V(G)$ be such that 
$M_X$ forms a perfect matching of $G[X]$,
and $P$ be an $M$-alternating path. 
\renewcommand{\labelenumi}{\theenumi}
\renewcommand{\labelenumi}{{\rm \theenumi}}
\renewcommand{\theenumi}{(\roman{enumi})}
\begin{enumerate} 
\item 
If $X\cap V(P)$ has no vertex exposed by $M_P$, 
then each connected component of $P[X]$ is 
an $M$-saturated path. 
\item If both end vertices of $P$ are in $X$, 
then each connected component of $P-E(G[X])$ is 
an $M$-ear relative to $X$.
\end{enumerate} 
\end{proposition}

%% file: barriers_prop_intersection_path.tex
\begin{proposition}\label{prop:intersection_path}
Let $G$ be a factorizable graph, and $M$ be a perfect matching of $G$. 
Let $P$ and $Q$ be $M$-alternating paths. 
If $P$ and $Q$ intersects only with their internal vertices, 
then each connected component of $P\cap Q$ is an $M$-saturated path. 
\end{proposition}

%% file: barriers_prop_two-element-extreme.tex
\begin{proposition}\label{prop:two-element-extreme}
Let $G$ be a factorizable graph and $M$ be a perfect matching of $G$, 
and let $u, v\in V(G)$. 
Then, the following two properties are equivalent:  
\renewcommand{\labelenumi}{\theenumi}
\renewcommand{\labelenumi}{{\rm \theenumi}}
\renewcommand{\theenumi}{(\roman{enumi})}
\begin{enumerate} 
\item The graph $G - u - v$ is  factorizable. 
\item There is an $M$-saturated path of $G$ between $u$ and $v$. 
\end{enumerate}
\end{proposition}

%% file: barriers_prop_circ2elem.tex
\begin{proposition}\label{prop:circ2elem}
Let $G$ be a factorizable graph and $M$ be a perfect matching of $G$, 
and let $u, v\in V(G)$. 
If there is an $M$-alternating circuit $C$ with $u, v\in V(C)$, 
then $u$ and $v$ are contained in the same 
factor-component of $G$. 
\end{proposition}

%% file: barriers_sec_pre_dm.tex
Factor-components of a bipartite factorizable graph 
are known to have the following partially ordered structure%
\footnote{This is different from the one in \cite{kita2012b, kita2012a}.
Though it is sometimes presented as a theorem for general bipartite graphs, 
we introduce it as one for bipartite factorizable graphs. }: 
%
\begin{theorem}%
[The Dulmage-Mendelsohn Decomposition~\cite{dm1958, dm1959, dm1963, lp1986, murota2000}]%
\label{thm:dm}
Let $G = (A, B; E)$ be a bipartite \matchablesp graph, 
and let $\mathcal{G}(G) =: \{G_i \}_{i\in I}$. 
Let $A_i:= A\cap V(G_i)$ and $B_i:= B\cap V(G_i)$ for each $i\in I$.
Then, there exists a partial order  $\pardm{A}$ on $\mathcal{G}(G)$ 
such that for any $i, j\in I$, 
\renewcommand{\labelenumi}{\theenumi}
\renewcommand{\labelenumi}{{\rm (\theenumi)}}
\renewcommand{\theenumi}{\roman{enumi}}
\begin{enumerate}
\item 
$E[B_j, A_i]\neq \emptyset$ 
 yields  $G_j \pardm{A} G_i$, and 
\item  
if $G_j \pardm{A} H \pardm{A} G_i$ yields $G_i = H$ or $G_j = H$ for any $H\in \mathcal{G}(G)$,
then $E[B_j, A_i]\neq \emptyset$.
\end{enumerate}
%
%
\end{theorem}
We call this decomposition of $G$ into a poset 
\textit{the Dulmage-Mendelsohn decomposition} (in short,  \textit{the DM-decomposition}), and 
each element of $\mathcal{G}(G)$, in this context, a \textit{DM-component}.
The DM-decomposition is uniquely determined by a graph, up to the choice of roles of 
color classes.
In this paper, we call the DM-decomposition of $G = (A, B; E)$
as in Theorem~\ref{thm:dm} the DM-decomposition \textit{with respect to }$A$.
%
\begin{proposition}[Dulmage \& Mendelsohn~\cite{dm1958, dm1959, dm1963, murota2000}]%
\label{prop:dm}
Let $G = (A, B; E)$ be a bipartite \matchablesp graph, and $M$ be a perfect matching of $G$.
Let $G_1, G_2 \in \mathcal{G}(G)$, 
and let $u\in A\cap V(G_1)$, $v\in A\cap V(G_2)$,
and $w\in B\cap V(G_2)$.
Then there is an $M$-balanced path from $u$ to $v$ 
if and only if $G_1\pardm{A} G_2$; 
additionally,  there is an $M$-saturated path between $u$ to $w$ 
if and only if $G_1\pardm{A} G_2$.
\end{proposition} 
%

%% file: cocoa2013_sec_motivation.tex
Given an elementary graph $G$, 
we say $u \sim v$ for $u, v\in V(G)$ 
if $u = v$ holds or $G-u-v$ is not factorizable. 
Kotzig~\cite{kotzig1959a, kotzig1959b, kotzig1960}
 found that $\sim$ is an equivalence relation. 
 Later Lov\'asz redefined it:  
 \begin{theorem}[Lov\'asz~\cite{lp1986}]%
\label{thm:canonicalpartition}
Let $G$ be an elementary graph. 
Then, the family of maximal barriers forms 
a partition of $V(G)$. 
Additionally, 
this partition coincides with the equivalence classes by $\sim$.  
\end{theorem}
This partition by the maximal barriers is called 
the {\em canonical partition}. 
As we mention in Section~\ref{sec:intro}, 
it plays fundamental and significant roles in matching theory. 
On the other hand, 
as for non-elementary graphs, 
the family of maximal barriers 
never forms a partition of the vertices (see \cite{lp1986}). 
The question remains: %
how all the maximal barriers exist 
and what is the counterpart in general graphs?
Therefore,  
we are going to investigate it. 
Actually, we work on a wider notion: {\em odd-maximal barriers}. %
\footnote{%
This is identical to those Kir\'{a}ly calls 
 strong barriers~\cite{kiraly1998}, 
 however we call it in the different way 
so as to avoid the confusion with the notion of strong end by Frank~\cite{frank1993}. }
\begin{definition}
Let $G$ be a graph. 
A barrier $X\subseteq V(G)$  is called an odd-maximal barrier if 
it is a barrier which is maximal with respect to $X\cup D_X$,  
i.e.,  
no $Y\subseteq D_X$ with $Y\neq \emptyset$
 satisfies that $X\cup Y$ is a barrier of $G$. 
\end{definition}
Odd-maximal barriers have some nice properties (see \cite{kiraly1998, kita2012f}):
First, A maximal barrier is an odd-maximal barrier. 
Second, for elementary graphs, 
the notion of maximal barriers and the notion of odd-maximal barriers coincide.
%
%
Hence, it seems reasonable to 
work on  the odd-maximal barriers.  

Given a graph $G$, 
we define $D(G)$ as the set of vertices 
that can be respectively exposed by maximum matchings,
$A(G)$ as $\Gamma(D(G))$
and $C(G)$ as 
$V(G)\setminus (D(G)\cup A(G))$. 
There is a well-known  theorem stating that 
$A(G)$ forms a barrier with special properties, 
called the {\em Gallai-Edmonds structure theorem}~\cite{lp1986}. 
Actually, with the Gallai-Edmonds structure theorem and the theorem by Kir\'aly~\cite{kiraly1998}, 
we can see that it suffices to work on factorizable graphs: 
\begin{proposition}[see also Kir\'aly~\cite{kiraly1998}]%
\label{prop:reduction2factorizable_nofoot}
Let $G$ be a graph. 
A set of vertices $S\subseteq V(G)$ is an odd-maximal barrier of $G$
if and only if it is a disjoint union of $A(G)$ and 
an odd-maximal barrier of the factorizable subgraph $G[C(G)]$. 
\end{proposition}
Given the above facts, in this paper 
we give canonical structures of odd-maximal barriers 
in general factorizable graphs  
that  can be regarded as a generalization of Lov\'asz's canonical partition, 
aiming to contribute to the foundation of matching theory. 

%% file: barriers_sec_canonical_short.tex
In this section we are going to introduce 
the canonical structure theorems of factorizable graphs, 
which shall serve as a language to describe odd-maximal barriers. 
They are composed of three parts: 
a partially ordered structure on the  factor-components 
(Theorem~\ref{thm:order}), 
a generalization of the canonical partition 
(Theorem~\ref{thm:generalizedcanonicalpartition}), 
and a relationship between these two (Theorem~\ref{thm:cor}).%
\footnote{All the statements in \cite{kita2012b} can be also found in 
\cite{kita2012a}}  
\begin{definition}
Let $G$ be a factorizable graph, 
and let $G_1, G_2 \in \mathcal{G}(G)$. 
We say $X\subseteq V(G)$ is a critical-inducing set for $G_1$ to $G_2$ 
if $X$ is separating, 
$V(G_1)\cup V(G_2)\subseteq X$ holds, 
and $G[X]/G_1$ is factor-critical. 
Additionally, 
we say $G_1 \yield G_2$ if there is a critical-inducing set 
for $G_1$ to $G_2$. 
\end{definition}
\begin{theorem}[Kita~\cite{kita2012a, kita2012b}]\label{thm:order}
For any factorizable graph $G$, 
$\yield$ is a partial order on $\mathcal{G}(G)$. 
\end{theorem}
\begin{definition}
Let $G$ be a factorizable graph. 
For $u, v\in V(G)$ 
we say $u\gsim{G} v$ if 
$u$ and $v$ are contained in the same factor-component of $G$, 
and $G-u-v$ is NOT factorizable. 
\end{definition}
\begin{theorem}[Kita~\cite{kita2012a, kita2012b}]\label{thm:generalizedcanonicalpartition}
For any factorizable graph $G$, 
$\gsim{G}$ is an equivalence relation on $V(G)$. 
\end{theorem}
As you can see by the definition, 
if $G$ is an elementary graph then $\sim $ and $\gsim{G}$ coincide. 
Therefore, 
we call the equivalence classes by $\gsim{G}$, i.e. $V(G)/\gsim{G}$,  
the {\em generalized canonical partition} or just the {\em canonical partition}, 
and denote by $\gpart{G}$. 
For each $H\in\mathcal{G}(G)$, 
we define $\pargpart{G}{H}:=\{ S\in\gpart{G}: S\subseteq V(H)\}$; 
then, $\pargpart{G}{H}$ forms a partition of $V(H)$, 
since  by the definition each equivalence class is respectively contained in 
one of the factor-components.  
Note that $\pargpart{G}{H}$ is always a refinement of $\gpart{H}$, 
which equals to $\pargpart{H}{H}$.  
%

For each $H\in \mathcal{G}(G)$, 
we denote the family of the upper bounds of $H$ 
in the poset $(\mathcal{G}(G), \yield)$ 
as $\parupstar{G}{H}$, 
and $\parupstar{G}{H}\setminus \{H\}$ as $\parup{G}{H}$. 
Moreover, 
we denote the vertices contained in $\parupstar{G}{H}$  as $\vparupstar{G}{H}$; 
i.e.,  $\vparupstar{G}{H}:=\bigcup_{H'\in\parupstar{G}{H}} V(H')$.  
We also denote $\vparupstar{G}{H}\setminus V(H)$ as $\vparup{G}{H}$. 
Actually, 
the next theorem states that 
each strict upper bound of $H \in \mathcal{G}(G)$ in $(\mathcal{G}(G), \yield)$ 
is respectively ``assigned'' to some $S\in\pargpart{G}{H}$: 
\begin{theorem}[Kita~\cite{kita2012a, kita2012b}]\label{thm:cor}
Let $G$ be a factorizable graph, 
and let $H\in \mathcal{G}(G)$. For each connected component $K$ of $G[\vparup{G}{H}]$, 
there exists $S_K\in\pargpart{G}{H}$ such that 
$\Gamma(K)\cap V(H) \subseteq S_K$. 
\end{theorem}
Based on Theorem~\ref{thm:cor}, 
we define $\parup{G}{S}$ as follows: 
$H'\in \parup{G}{S}$ if and only if 
$H\yield H'$ and  $H\neq H'$ holds   
and there exists a connected component $K$ of $G[\vup{H}]$ with 
$\Gamma(K)\cap V(H) \subseteq S$ such that $V(H')\subseteq V(K)$. 
Additionally, 
we denote the vertices contained in $\parup{G}{S}$ as $\vparup{G}{S}$; i.e.,  
$\vparup{G}{S}: = \bigcup_{H'\in\parup{G}{S}} V(H')$.   
We also define $\vparupstar{G}{S}:= \vparup{G}{S} \cup S$. 
Regarding these eight notations 
we sometimes omit the subscripts ``$G$'' if they are apparent from the contexts. 
Note that $\dot{\bigcup}_{T\in \pargpart{G}{H}} \up{T}  = \up{H}$.  

We call the canonical structures of factorizable graphs 
given by Theorems~\ref{thm:order}, \ref{thm:generalizedcanonicalpartition}, 
and \ref{thm:cor}
the {\em generalized cathedral structures} 
or just the {\em cathedral structures}. 
Now let us add some propositions used later in this paper: 
\begin{proposition}[Kita~\cite{kita2012a, kita2012b}]\label{prop:block}
Let $G$ be a factorizable graph, 
and let $H\in \mathcal{G}(G)$. 
Then, $G[\vupstar{H}]/H$ is factor-critical, 
so is each block of it. 
\end{proposition}

\begin{proposition}[Kita~\cite{kita2012a, kita2012b}]\label{prop:ear}
Let $G$ be a factorizable graph and $M$ be a perfect matching of $G$, 
and let $H\in\mathcal{G}(G)$. 
Let $P$ be an $M$-ear relative to $H$. 
\renewcommand{\labelenumi}{\theenumi}
\renewcommand{\labelenumi}{{\rm \theenumi}}
\renewcommand{\theenumi}{(\roman{enumi})}
\begin{enumerate}
\item
Let $H'\in\mathcal{G}(G)$.  
If $P$ is through $H'$,  
then $H\yield H'$. 
\item 
The end vertices $u, v\in V(H)$ of $P$ satisfies $u\gsim{G} v$. 
\end{enumerate}
\end{proposition}

%% file: cocoa2013_sec_maintheorem.tex
Our main result is the following: 
\begin{maintheorem}\label{thm:main}
Let $G$ be a factorizable graph, 
and $X\subseteq V(G)$ be an odd-maximal barrier of $G$. 
Then, 
$X$ is a disjoint union of some members of $\gpart{G}$; 
namely, 
there exists $S_1,\ldots, S_k\in\gpart{G}$ such that 
$X = S_1\dot{\cup} \cdots \dot{\cup}S_k$.  
Additionally, 
odd components of $G-X$ have structures as follows: 
$D_X = (\vupstar{G_1}\setminus\vupstar{S_1}) 
\dot{\cup}\cdots\dot{\cup}
(\vupstar{G_k}\setminus\vupstar{S_k})$, 
where $G_i\in\mathcal{G}(G)$ is such that $S_i\in\pargpart{G}{G_i}$ 
for each $i\in\{1,\ldots,k\}$. 
\end{maintheorem}
This theorem states that 
in general graphs 
the equivalence classes of the generalized canonical partition 
are the ``atoms'' that constitute odd-maximal barriers, 
and that odd components associated to odd-maximal barriers are also described canonically by 
the generalized cathedral structure. 
As we see in previous sections, 
among two formulations of the canonical partition of elementary graphs, 
the generalization of the canonical partition introduced in \cite{kita2012a, kita2012b} 
is  attained based on Kotzig's formulation; 
here we show it is as well a generalization based on Lov\'asz's formulation. 

This theorem is an immediate corollary of Theorem~\ref{thm:barrier2up}, 
and the rest of this paper is to prove Theorem~\ref{thm:barrier2up}. 
We shall prove it 
 by examining the reachability of alternating paths 
from two viewpoints---
regarding odd-maximal barriers and regarding the generalized cathedral structure---%
and showing their equivalence. 
Let us mention an additional property used later in this paper. 
\begin{proposition}[Kir\'aly~\cite{kiraly1998}] \label{prop:odd-maximal}
A barrier $X\subseteq V(G)$ of a graph $G$ is odd-maximal 
if and only if 
all the odd components of $G-X$ are factor-critical. 
\end{proposition}

%% file: barriers_prop-proof_extreme2path_x2d_closed.tex
\begin{proposition}[folklore]\label{prop:extreme2path}
Let $G$ be a factorizable graph, and $M$ be a perfect matching of $G$. 
If $X\subseteq V(G)$ is a barrier, then  
for any $u, v\in X$ there is no $M$-saturated path between $u$ and $v$. 
\end{proposition}
\begin{proof}
Suppose the claim fails, namely, 
there is an $M$-saturated path, say $P$,  between 
vertices $u$ and $v$.  
Then, $M\Delta E(P)$, i.e., $(M\setminus E(P))\cup (E(P)\setminus M)$,  
forms a perfect matching of $G-u-v$; 
accordingly, $G-u-v$ is factorizable. 
Now recall that 
since $X$ is a barrier of the factorizable graph $G$, 
$G-X$ has exactly $|X|$ odd components by the definition of barriers. 
Therefore, 
the graph $(G-u-v)-(X\setminus \{u, v\})$, 
which equals to $G - X$, 
also has $|X|$ odd components; 
this means by the Berge formula that 
 $G-u-v$ is not factorizable, a contradiction. 
\qed
\end{proof}

\begin{proposition}[folklore]\label{prop:x2d}
Let $G$ be a factorizable graph,  $M$ be a perfect matching, 
and $X$ be a barrier of $G$. 
Then, for any $x\in X$ and $y\in C_X$, 
there is no $M$-saturated path between $x$ and $y$ 
nor $M$-balanced path from $x$ to $y$. 
\end{proposition}
\begin{proof}
Suppose otherwise, that is, there is a path $P$ 
which is $M$-saturated between $x$ and $y$
 or $M$-balanced from $x$ to $y$. 
Trace $P$ from $x$;  let $z$ be the first vertex we encounter that is in $C_X$, 
and let $w$ be the last vertex in $X\cup D_X$ we encounter 
if we trace $xPz$ from $x$. 
%
Apparently, $w\in X$ and  $wz\in E(xPz)$ hold, and  
by Proposition~\ref{prop:barrier},  $wz\not\in M$ holds.
Therefore,   
$xPw$ is an $M$-saturated path between $x$ and $w$, 
contradicting Proposition~\ref{prop:extreme2path}.
\qed
\end{proof}

%% file: cocoa2013_prop_x2xd.tex
\begin{proposition}[folklore]%
\label{prop:x2xd}
Let $G$ be a factorizable graph, $M$ be a perfect matching, 
and $X\subseteq V(G)$ be an odd-maximal barrier. 
Then, for any $u\in X$ and $v\in X\cup C_X$ 
there is no $M$-saturated path between $u$ and $v$. 
\end{proposition}

%% file: cocoa2013_proof_prop_x2xd.tex
\begin{proof}
This is immediate by Proposition~\ref{prop:extreme2path} 
and Proposition~\ref{prop:x2d}. 
\qed
\end{proof}

%% file: cocoa2013_def_hxg.tex
\begin{definition}
Let $G$ be a graph, $X\subseteq V(G)$, 
and $K_1,\ldots, K_l$ be the odd components of $G-X$. 
We denote  the bipartite graph 
resulting from deleting the even components of $G-X$, 
removing the edges whose vertices are all contained in $X$,  and 
contracting each $K_i$, where $i = 1,\ldots, l$,  respectively into one vertex,
as $\hxg{X}{G}$. 
Namely, $\hxg{X}{G} := (G - C_X - E(G[X]))/K_1/\cdots/K_l$. 
\end{definition}

%% file: cocoa2013_prop_projection.tex
\begin{proposition}[might be a folklore]\label{prop:projection}
Let $G$ be a factorizable graph and $X$ be an  odd-maximal barrier of $G$.
If $M\subseteq E(G)$ is a perfect matching of $G$, then 
$M\cap \delta(X)$ forms a perfect matching of $\hxg{X}{G}$.
Conversely, 
 if $M'$ is a perfect matching of $\hxg{X}{G}$,
there is a perfect  matching $M$ of $G$ such that 
$M'= M\cap \delta(X)$.
\end{proposition}

%% file: cocoa2013_proof_prop_projection.tex
\begin{proof}
The first claim follows by Proposition~\ref{prop:barrier}. 
For the second claim, 
first note that $G[C_X]$ is factorizable by Proposition~\ref{prop:barrier}, 
and let $N$ be a perfect matching of $G[C_X]$. 
By Proposition~\ref{prop:odd-maximal}, 
the odd components $K_1,\ldots, K_l$ of $G-X$ are 
each factor-critical. 
For each $i = 1,\ldots, l$ let $M_i$ be a near-perfect matching of $K_i$ 
exposing only the vertex covered by $M'$. 
Then, $N\cup M'\cup \bigcup_{i=1}^l M_i$ forms a desired perfect matching. 
\qed
\end{proof}

%% file: cocoa2013_prop_pathprojection.tex
\begin{proposition}\label{prop:pathprojection}
Let $G$ be a factorizable graph, $X\subseteq V(G)$ be an odd-maximal barrier of $G$, 
and $\mathcal{K} := \{K_i\}_{i=1}^l$    
be the family of odd components of $G-X$, where $l = |X|$. 
Let $M$ be a perfect matching of $G$, and 
$M'$ be the perfect matching of $\hxg{X}{G}$ such that $M' = M \cap \delta(X)$. 
Let $u, v \in X$,  and $w\in V(K)$, where $K \in \mathcal{K} \if0\{K_i\}_{i=1}^l\fi$, 
and let $w_K$ be the contracted vertex of $\hxg{X}{G}$ corresponding to $K$. 
\renewcommand{\labelenumi}{\theenumi}
\renewcommand{\labelenumi}{{\rm \theenumi}}
\renewcommand{\theenumi}{(\roman{enumi})}
\begin{enumerate} 
\item \label{item:projection}
Then, for any $M$-balanced path (resp. $M$-saturated path) $P$ of $G$ 
from $u$ to $v$ (resp. between $u$ and $w$), 
$P' = P/K_1/\cdots/K_l$ is an $M'$-balanced path 
(resp. $M'$-saturated path) of $\hxg{X}{G}$
from $u$ to $v$ (resp. between $u$ and $w_K$). 
\item \label{item:expansion} 
Conversely, 
for any $M'$-balanced path (resp. $M'$-saturated path) $P'$ 
 from $u$ to $v$ in $\hxg{X}{G}$
 (resp. between $u$ and $w_K$),   
there is an $M$-balanced path (resp. $M$-saturated path) $P$
 from $u$ to $v$ in $G$ (resp. between $u$ and $w$)  
such that $P' = P/K_1/\cdots/K_l$.
\end{enumerate}
\end{proposition}

%% file: cocoa2013_proof_prop_pathprojection.tex
\begin{proof}
For \ref{item:projection},
 we first prove the case where $P$ is an $M$-balanced path.  
Let $u = x_1, \ldots, x_{l'} = v$ be the vertices of $X\cap V(P)$, 
and suppose, without loss of generality, they appear in this order 
if we trace $P$ from $u$.  
For each $i = 1, \ldots, l'$, 
let $L_i\in \mathcal{K}$ be such that 
$x_i'\in V(L_i)$, 
which is well-defined by Proposition~\ref{prop:barrier}, and  
let $z_i$ be the contracted vertex of $\hxg{X}{G}$ 
corresponding to $L_i$. 
Note that by Proposition~\ref{prop:barrier}, 
\begin{cclaim}\label{claim:distinct}
if $x_i\neq x_j$, then $L_i\neq L_j$, accordingly, $z_i\neq z_j$. 
\end{cclaim}
%
We are going to prove a bit refined statement, 
\begin{quotation}
\noindent
$P'$ is an $M'$-balanced path from $u$ to $v$, with
$V(P') = \{x_i\}_{i=1}^{l'} \cup \{z_i\}_{i=1}^{l'}\setminus \{z_{l'}\}$,  
\end{quotation}
by induction on $k$, where $|E(P)|=: 2k$. 

If $k=0$, then the statement is obviously true. 
Let $k > 0$ and suppose the claim is true for $0, \ldots, k-1$. 
By the definitions, 
the internal vertices of $uPx_2$ are contained in $L_1$. 
By Proposition~\ref{prop:barrier}, 
$\delta(L_i)\cap M = \{uu'\}$. 
Thus, if we trace $uPx_2$ from $u$ then the last edge is not in $M$, 
which means $uPx_2$ is an $M$-balanced path from $u$ to $x_2$. 
Accordingly,  $x_2Pv$ is also an $M$-balanced path, from $x_2$ to $v$.

Note that $P_1':= uPx_2/K_1/\ldots/K_l$ is apparently an $M$-balanced path, 
since $E(P_1') = \{uz_1, z_1x_2\}$. 
Therefore, if $x_2 = v$ then the claim follows. 
Hence hereafter we prove the case where $x_2\neq v$. 
By the induction hypothesis, 
$P_2' := x_2Pv/K_1/\ldots/K_l$ is an $M'$-balanced paths of $\hxg{X}{G}$, 
whose vertices are $\{ x_2,\ldots, x_{l'}= v\} \cup \{z_2,\ldots, z_{l'-1}\}$. 
Thus,  $V(uPx_2)\cap V(x_2Pv) = \{x_2\}$ by Claim~\ref{claim:distinct};  
accordingly, $P' = P_1' + P_2'$ is an $M'$-balanced path of $\hxg{X}{G}$
from $u$ to $v$ with $V(P') = \{x_i\}_{i=1}^{l'}\cup \{z_i\}_{i=1}^{l'} \setminus \{z_{l'}\}$. 
The other case where $P$ is an $M$-saturated path 
can be proved by similar arguments.

For \ref{item:expansion}, 
we first prove the case where $P'$ is an $M'$-balanced path of $\hxg{X}{G}$. 
Since it is apparently true if $u = v$, 
we prove the case where $u \neq v$. 
Let $u = x_0, y_0,  \ldots, x_p, y_p, x_{p+1} = v$ be the vertices of $P'$,
and suppose they appear in this order if we trace $P'$ from $u$. 
Note that $x_i\in X$ for each $i = 0,\ldots, p+1$  
 and that $y_i$ be a contracted vertex 
corresponding to an odd component of $G-X$, say $L_i$,
 for each $0,\ldots ,p$.   
For each $i=0,\ldots, p$, 
let $y_i^1, y_i^2 \in V(L_i)$ be such that 
$G$ has edges $x_iy_i^1,  y_i^2x_{i+1} \in E(G)$ 
that correspond to  $x_iy_i, y_ix_{i+1} \in E(\hxg{X}{G})$, respectively. 
Since $x_iy_i\in M'$ and $y_ix_{i+1}\not\in M'$, 
it follows that $x_iy_i^1\in M$ and $y_i^2x_{i+1}\in M$. 
%
%
%
%
%
The odd component $L_i$ is factor-critical by Proposition~\ref{prop:odd-maximal}, 
and $M_{L_i}$ forms a near-perfect matching of $L_i$, 
which exposes $y_i^1$, by Proposition~\ref{prop:barrier}. 
Therefore, there is an $M$-balanced path $Q_i$ 
from $y_i^2$ to $y_i^1$ which is  contained in $L_i$,
 by Proposition~\ref{prop:path2root}. 
Thus, replacing each $y_i$ by $Q_i$ on $P'$,
we can get an $M$-balanced path from $u$ to $v$ in $G$. 
The other case where $P'$ is an $M'$-saturated path 
can be proved by similar arguments. 
\qed
\end{proof}  

%% file: cocoa2013_comment_color.tex
Given a factorizable graph $G$ 
 and an odd-maximal barrier $X$, 
we denote the DM-decomposition of $\hxg{X}{G}$ with respect to $X$ 
as just the DM-decomposition of $\hxg{X}{G}$. 
In this case, we sometimes denote $\dm_X$ as just $\dm$, 
omitting the subscript ``$X$''.  

%% file: cocoa2013_def_expansion.tex
\begin{definition}
Let $G$ be a factorizable graph, and $X$ be an odd-maximal barrier of $G$. 
Let $D$ be a DM-component of $\hxg{X}{G}$, 
whose vertices in $V(D)\setminus X$  are 
the contracted vertices 
resulting from some odd components of $G-X$, say $K_1,\ldots, K_l$, where $l \le |X|$.
We say $\what{D}$ is the \textit{expansion} of $D$ 
if it is the subgraph of $G$ induced by 
$\left( V(D)\cap X \right) \cup \bigcup_{i=1}^l V(K_i)$. 
\end{definition}

%% file: cocoa2013_prop_dm2partition.tex
\begin{proposition}\label{prop:dm2partition}
Let $G$ be a factorizable graph, and $X$ be an odd-maximal barrier of $G$. 
Let $D_1,\ldots, D_k$ be the DM-components of $\hxg{X}{G}$. 
For each $i = 1,\ldots, k$,  
let $\what{D}_i$ be the expansion of $D_i$.  
Then, 
\renewcommand{\labelenumi}{\theenumi}
\renewcommand{\labelenumi}{{\rm \theenumi}}
\renewcommand{\theenumi}{(\roman{enumi})}
\begin{enumerate}
\item \label{item:partition}
$\{V(\what{D}_i)\}_{i=1}^k$ forms a partition of $X\cup D_X$, 
\item \label{item:separating}
$V(\what{D}_i)$ is separating, accordingly $\what{D}_i$ is factorizable, 
\item \label{item:barrier}
$X\cap V(\what{D}_i)$ is an odd-maximal barrier of $\what{D}_i$, and 
\item \label{item:isomorphic}
$\hxg{X\cap V(\what{D}_i)}{\what{D}_i}$ is isomorphic to $D_i$, 
for each $i = 1,\ldots, k$. 
\end{enumerate} 
\end{proposition}

%% file: cocoa2013_proof_prop_dm2partition.tex
\begin{proof}
Since the DM-components of $\hxg{X}{G}$ 
give the partition of the vertices of it, 
\ref{item:partition} apparently follows from the definitions. 
For the first half of \ref{item:separating}, 
suppose that $V(\what{D}_i)$ is not separating, equivalently 
by Proposition~\ref{prop:separating}, 
that there is a perfect matching $M$ of $G$ such that 
$\delta(\what{D}_i)\cap M \neq \emptyset$. 
Then, by Proposition~\ref{prop:projection}, 
$M' := M\cap \delta(X)$ 
forms a perfect matching of $\hxg{X}{G}$  satisfying 
$\delta(D_i)\cap M' \neq \emptyset$, 
a contradiction. 
Therefore,  $V(\what{D}_i)$ is separating;  
accordingly,  $\what{D}_i$ is factorizable, 
 and we are done for \ref{item:separating}.  
 
By the definition, 
$\what{D}_i\setminus X$ is composed of $|X\cap V(\what{D}_i)|$ number of 
odd-components, each of which is factor-critical by Proposition~\ref{prop:odd-maximal}. 
Therefore,  
$V(\what{D}_i)\cap X$ is an odd-maximal barrier of $\what{D}_i$ 
by the statement \ref{item:separating} and Proposition~\ref{prop:odd-maximal} again. 
Thus, we are done for \ref{item:barrier}. 
The statement \ref{item:isomorphic} is apparent from the definitions.   
\qed
\end{proof}

%% file: cocoa2013_thm_path2dm.tex
\begin{theorem}\label{thm:path2dm}
Let $G$ be a factorizable graph, $X$ be an odd-maximal barrier, 
and $M$ be a perfect matching of $G$.  
Let $u, v\in X$, and $w \in D_X$, and for each $\alpha = u, v, w$ let  
$D_{\alpha}$ be the DM-component of $\hxg{X}{G}$ 
whose expansion $\what{D}_\alpha$ contains $\alpha$. 
Then, 
there is an $M$-balanced path from $u$ to $v$ 
(resp.  an $M$-saturated path from $u$ to $w$) in $G$ 
if and only if 
$D_u \dm D_v$ (resp.  $D_u \dm D_w$). 
\end{theorem}

%% file: cocoa2013_proof_thm_path2dm.tex
\begin{proof}
First note that $\what{D}_\alpha$ is well-defined by Proposition~\ref{prop:dm2partition}. 
Now the claim is immediate from 
Proposition~\ref{prop:dm} and Proposition~\ref{prop:pathprojection}.
\qed
\end{proof}

%% file: cocoa2013_lem_bip-traverse.tex
\begin{lemma}\label{lem:bip-traverse}
Let $G = (A, B; E)$ be a bipartite factorizable graph, 
$M$ be a perfect matching of $G$,  
and $D_1, D_2$ be DM-components of $G$ with $D_1\pardm{A} D_2$. 
Then, for any $u \in V(D_1)\cap A$ and $v\in V(D_2)\cap B$, 
any $M$-saturated path $P$ between $u$ and $v$ 
traverses $A\cap V(D_2)$. 
\end{lemma}

%% file: cocoa2013_proof_lem_bip-traverse.tex
\begin{proof}
Apparently $vv'\in E(P)$ and $v'\in V(D_2)\cap A$;   
therefore, the claim follows. 
\qed
\end{proof}

%% file: cocoa2013_lem_traverse.tex
\begin{lemma}\label{lem:traverse}
Let $G$ be a factorizable graph, $X$ be an odd-maximal barrier, 
and $M$ be a perfect matching of $G$.  
Let $\what{D}_1$ and $\what{D}_2$ be the subgraphs of $G$ 
which are respectively the expansions of 
 DM-components $D_1$ and $D_2$ such that $D_1\dm D_2$. 
Then, for any $u\in X\cap V(\what{D}_1)$ and $w\in V(\what{D}_2)\setminus X$, 
any $M$-saturated path $P$ between $u$ and $w$ 
traverses $X\cap V(\what{D}_2)$. 
\end{lemma}

%% file: cocoa2013_proof_lem_traverse.tex
\begin{proof}
Let $K_1,\ldots, K_l$, where $l = |X|$,  be the odd components of $G-X$. 
By Proposition~\ref{prop:pathprojection}, 
$P' := P/K_1/\cdots/K_l$ is an $M'$-saturated path, 
where $M' = M\cap \delta(X)$, 
whose end vertices are respectively in $X\cap V(D_1)$ and $V(D_2)\setminus X$. 
Therefore, $P'$ traverses $X\cap V(D_2)$ by Lemma~\ref{lem:bip-traverse}, 
which means $P$ traverses $X\cap V(\what{D}_2)$. 
\qed
\end{proof}

%% file: cocoa2013_prop_base.tex
\begin{proposition}[implicitly stated in \cite{kita2012b, kita2012a}]%
\label{prop:base}
Let $G$ be a factorizable graph,  and let $H\in\mathcal{G}(G)$ 
and $S\in \pargpart{G}{H}$. 
Then, $G[\vupstar{S}]/S$ is factor-critical. 
\end{proposition}

%% file: cocoa2013_lem_up-path.tex
\begin{lemma}\label{lem:up-path}
Let $G$ be a factorizable graph and  $M$ be a perfect matching of $G$, 
and let $H\in\mathcal{G}(G)$ 
and $S\in \pargpart{G}{H}$. 
Then, 
for any $x\in \vupstar{S}$, 
there is an $M$-balanced path from $x$ to some vertex $y\in S$, 
whose vertices except $y$ are contained in $\vup{S}$. 
\end{lemma}

%% file: cocoa2013_proof_lem_up-path.tex
\begin{proof}
$M_{\vup{S}}$ forms a near-perfect matching of $G' := G[\vupstar{S}]/S$ 
exposing only the contracted vertex $s$ corresponding to $S$, 
and by Proposition~\ref{prop:base}, 
$G'$ is factor-critical. 
Therefore, by Proposition~\ref{prop:path2root} 
there is an $M_{\vup{S}}$-balanced path from any $x\in \vupstar{S}$ to $s$, 
which corresponds to a desired path in $G$. 
Thus, the claim follows. 
\qed
\end{proof}

%% file: cocoa2013_prop_deletable2path.tex
%
\begin{proposition}\label{prop:deletable2path}\label{prop:maximal}
Let $G$ be a factorizable graph and  $M$ be a perfect matching of $G$, 
and let $H\in\mathcal{G}(G)$. 
A set of vertices $S\subseteq V(H)$ 
is a member of $\pargpart{G}{H}$ if and only if 
it is a maximal subset of $V(H)$ satisfying that 
there is no $M$-saturated path between any two vertices of it. 
\end{proposition}

%% file: cocoa2013_proof_prop_maximal.tex
\begin{proof}
This follows easily from Theorem~\ref{thm:generalizedcanonicalpartition} 
and Proposition~\ref{prop:two-element-extreme}. 
\qed
\end{proof}

%% file: cocoa2013_lem_nosaturate.tex
\begin{lemma} \label{lem:nosaturate}
Let $G$ be a factorizable graph and $M$ be a perfect matching of $G$, 
and let $H\in \mathcal{G}(G)$ and $S\in \pargpart{G}{H}$.
Then, for any $s\in S$ and $x\in \vup{S}$, 
there is no $M$-saturated path between $s$ and $x$
nor $M$-balanced path from $s$ to $x$.
\end{lemma}

%% file: cocoa2013_proof_lem_nosaturate.tex
\begin{proof}
Suppose the claim fails, that is,
there is a path $P$ that is 
$M$-balanced from $s$ to $x$ or $M$-saturated between $s$ and $x$.  
Trace $P$ from $s$ and let $y$ be the first vertex we encounter that is in $\vup{S}$.
Trace $sPy$ from $y$ and let $z$ be the first vertex we encounter that is in $V(H)$. 
Then, since $V(H)$ and $\vup{S}$ are separating, 
$zPy$ is an $M$-exposed path by Proposition~\ref{prop:separating}. 
Consequently $sPz$ is an $M$-saturated path between $s$ and $z$, 
which means $z\not\in S$ by Proposition~\ref{prop:deletable2path}. 
%

On the other hand, by Lemma~\ref{lem:up-path}, 
there is an $M$-balanced path $Q$ from $y$ to some vertex $t\in S$ 
whose vertices except $t$ are contained in $\vup{S}$. 
Therefore, $zPy + yQt$ is an $M$-ear relative to $H$, 
whose end vertices are $z$ and $t$;  
this contradicts Proposition~\ref{prop:ear} since $z\not\gsim{G} t$. 
%
%
%
%
\qed  
\end{proof}

%% file: cocoa2013_lem_distinctivebase.tex
\begin{lemma}\label{lem:distinctivebase}
Let $G$ be a factorizable graph and $M$ be a perfect matching of $G$.  
Let $H\in \mathcal{G}(G)$, and  let $u, v \in V(H)$ be such that 
$u\not\gsim{G} v$. 
Let $P$ be an $M$-saturated path between $u$ and $v$ 
such that $E(P)\setminus E(H) \neq \emptyset$, 
and let $P_1,\ldots , P_l$ be the components of $P-E(H)$.
Let $S_0, S_{l+1} \in \pargpart{G}{H}$ be such that $u\in S_0$ and $v\in S_{l+1}$.
Then,  
\renewcommand{\labelenumi}{\theenumi}
\renewcommand{\labelenumi}{{\rm \theenumi}}
\renewcommand{\theenumi}{(\roman{enumi})}
\begin{enumerate}
\item \label{item:base} 
two end vertices of $P_i$ belong to the same member of $\pargpart{G}{H}$,
 say $S_i$, 
\item \label{item:up} 
$P_i$ is, except its end vertices, contained in $\vup{S_i}$
for each $i = 1,\ldots, l$, and 
\item \label{item:distinctive}
for any $i,j\in\{0,\ldots, l+1\}$ with $i\neq j$, 
$S_i\neq S_j$. 
\end{enumerate}
\end{lemma}

%% file: cocoa2013_proof_lem_distinctivebase.tex
\begin{proof}
By Proposition~\ref{prop:intersection}, 
$P_i$ is an $M$-ear relative to $H$ for each $i = 1,\ldots, l$;  
therefore,  
\ref{item:base} follows by Proposition~\ref{prop:ear}.
Thus, \ref{item:up} follows by Proposition~\ref{prop:ear}. 
For \ref{item:distinctive},
let the end vertices of $P_i$ be $x_i$ and $y_i$ for each $i = 1,\ldots, l$.
Without loss of generality, we can assume that
the vertices $u =: y_0, x_1, y_1,\ldots, x_l, y_l,x_{l+1}:= v$ appear in this order 
if we trace $P$ from $u$.
Then,  for any $i, j$ with $0\le i < j \le l+1$,
$y_iPx_j$ forms an $M$-saturated path between $y_i\in S_i$ and $x_j\in S_j$. 
Thus we have $S_i\neq S_j$ by Proposition~\ref{prop:deletable2path}; 
this means \ref{item:distinctive}, and we are done. 
\qed
\end{proof}

%% file: cocoa2013_lem_path2up.tex
\begin{lemma}\label{lem:path2up}
Let G be a factorizable graph and  $M$ be a perfect matching of $G$. 
Let $H\in\mathcal{G}(G)$,  and let 
$S, T\in\pargpart{G}{H}$ be such that  $S\neq T$.  
Then, for any $s\in S$ and $t\in \vupstar{T}$, 
 there is an $M$-saturated path $P$ between $s$ and $t$,
which is contained in $\vupstar{H}\setminus \vup{S}$.
\end{lemma}

%% file: cocoa2013_proof_lem_path2up.tex
\begin{proof}
By Lemma~\ref{lem:up-path},
there is  an $M$-balanced path $P_1$ from $t$ to a vertex $x\in T$ 
whose vertices except $x$ are  contained in $\vup{T}$. 
By Proposition~\ref{prop:deletable2path}, 
there is an $M$-saturated path $P_2$ between $s$ and $x$. 
By Lemma~\ref{lem:distinctivebase}, 
$V(P_2)$ is contained in $\vupstar{H}\setminus \vup{S}\setminus \vup{T}$;  
accordingly,  $V(P_1)\cap V(P_2) = \{ x\}$.
Hence, $P:= P_1 + P_2$ is an $M$-saturated path between $s$ and $t$,
contained in $\vupstar{H}\setminus \vup{S}$.
\qed 
\end{proof}

%% file: cocoa2013_lem_path2up_restated.tex
\begin{lemma}\label{lem:path2up_restated}
Let G be a factorizable graph and $M$ be a perfect matching of $G$.  
Let $H\in\mathcal{G}(G)$,  and 
let $S, T\in\pargpart{G}{H}$ be such that  $S\neq T$.  
Then, for any $s\in S$ and $t\in \vupstar{T}$, 
 there is an $M$-saturated path $P$ between $s$ and $t$
 such that for any $u \in S$ and $v\in V(P)\setminus S$
 there is an $M$-saturated path between $u$ and $v$. 
\end{lemma}

%% file: cocoa2013_thm_char2elem.tex
\begin{theorem}\label{thm:char2elem}
Let $G$ be a factorizable graph, $M$ be a perfect matching of $G$,
and $u, v\in V(G)$ be such that 
$G-u-v$ is not factorizable. 
If there are $M$-balanced paths respectively  from $u$ to $v$ 
and from $v$ to $u$,
then $u$ and $v$ are in the same factor-component of $G$.
\end{theorem}

%% file: cocoa2013_proof_thm_char2elem.tex
\begin{proof}
Let $P$ be an $M$-balanced path from $u$ to $v$,
and $Q$ be an $M$-balanced path from $v$ to $u$. 
Let $x_0, x_1, \ldots $ be the sequence of vertices in $V(P)\cap V(Q)$ defined 
by the following procedure:  
 
\begin{algorithmic}[1]
\STATE $x_0:= v$; $i:= 0$; 
\WHILE{$x_i\neq u$} 
\STATE trace $x_iQu$ from $x_i$ and let $x_{i+1}$ be the first vertex we encounter 
that is in $V(uPx_i)\setminus \{x_i\}$;  
\STATE $i++$. 
\ENDWHILE
\end{algorithmic}

Note that this procedure surely stops in finite time 
(since each repetition of the while-loop $x_i$ draw nearer to $u$) 
and returns $v=x_0,\ldots, x_l=u$ for some $l\ge 0$. 
Note also the next claim, which is easy to see by the definition. 
\begin{cclaim}\label{claim:commonvertex} 
\renewcommand{\labelenumi}{\theenumi}
\renewcommand{\labelenumi}{{\rm \theenumi}}
\renewcommand{\theenumi}{(\roman{enumi})}
\begin{enumerate}
\item \label{item:reverse}
Tracing $P$ from $u$, we encounter $x_l, \ldots, x_0$ in this order. 
\item \label{item:commontwo}
For each $i = 0,\ldots, l-1$, 
$uPx_i$ and $x_iQx_{i+1}$ have only $\{x_i, x_{i+1}\}$ as common vertices. 
\item \label{item:commonone}
For each $i = 0,\ldots, l$,  
$uPx_i$ and $vQx_i$ has only $x_i$ as a common vertex. 
\end{enumerate}
\end{cclaim}
\begin{proof}
By the definition procedure, for each $i = 0, \ldots, l-1$, 
$x_{i+1}$ is located on $P$ nearer to $u$ than $x_i$ is;  
this yields \ref{item:reverse}. 
The statement \ref{item:commontwo} is also apparent from the definition. 

For \ref{item:commonone} note that 
$vQx_i = x_0Qx_1 + \cdots + x_{i-1}Qx_i$.  
Therefore it suffices to prove that for each $0 \le j \le i-1$ 
$uPx_i$ and $x_{j}Qx_{j+1}$ have at most $x_i$ as a common vertex; 
this holds true, 
since $V(uPx_j)\cap V(x_jQx_{j+1}) = \{x_j, x_{j+1} \}$ by \ref{item:commontwo}, 
and $V(uPx_i) \subseteq V(uPx_j)\setminus \{x_j, \ldots, x_{i-1}\}$ by \ref{item:reverse}. 
\qed
\end{proof}

\begin{cclaim}\label{claim:char2elem}
For each $i = 0, \ldots, l-1$, 
$x_iQx_{i+1}$ is an $M$-balanced path from $x_i$ to $x_{i+1}$. 
For each $i = 0,\ldots, l$, 
$uPx_i$ and $vQx_i$ are $M$-balanced paths
 from $u$ to $x_i$ and from $v$ to $x_i$, respectively.  
\end{cclaim}
\begin{proof}
We give it by the induction on $i$. 
If $i=0$ then both of the claims are rather trivially true, 
and if $i = l$ then the second claim is trivially true. 
Therefore 
let  $0 < i < l$ and suppose the claims are true for $i-1$. 
Since $vQx_{i-1}$ is an $M$-balanced path from $v$ to $x_{i-1}$ by the induction hypothesis, 
$x_{i-1}Qu$ is an $M$-balanced path from $x_{i-1}$ to $u$. 
Additionally, since $uPx_{i-1}$ is an $M$-balanced path from $u$ to $x_{i-1}$
by the induction hypothesis, 
it follows that  $x'_{i-1}\in V(x_{i-1}Pv)$ 
and 
the definition procedure yields that 
$x_{i-1}Qx_i$ is an $M$-balanced path from $x_{i-1}$ to $x_i$. 
Therefore, we have that $vQx_i = vQx_{i-1} + x_{i-1}Qx_i$ 
is also an $M$-balanced path, from $v$ to $x_i$. 

Now note that $uPx_{i} + x_iQv$ forms a path, 
since they have only $x_{i}$ as a common vertex by Claim~\ref{claim:commonvertex}. 
Suppose that $uPx_i$ is an $M$-saturated path between $u$ and $x_i$. 
Then, $uPx_i + x_iQv$ is an $M$-saturated path between $u$ and $v$. 
This contradicts Proposition~\ref{prop:extreme2path}. 
Therefore, $uPx_i$ is an $M$-balanced path from $u$ to $x_i$, 
and we are done.   
\qed
\end{proof}
Since Claim~\ref{claim:char2elem} says
$uPx_i$ is an $M$-balanced path for each $i = 0,\ldots, l$, 
it follows by Claim~\ref{claim:commonvertex}  
 that $x_iPx_{i+1}$ is an $M$-balanced path from 
$x_{i+1}$ to $x_i$ for each $i = 0, \ldots, l-1$. 
Therefore, 
$x_iQx_{i+1}$ and $x_{i+1}Px_i$ forms 
an $M$-alternating circuit, since 
they have only $\{x_i, x_{i+1}\}$ as common vertices by Claim~\ref{claim:commonvertex}. 
Therefore, by Proposition~\ref{prop:circ2elem}, 
$x_i$ and $x_{i+1}$ are contained in the same factor-component of $G$
for each $i= 0,\ldots, l-1$. 
This yields that $u$ and $v$ are contained in the same factor-component. 
\qed
\end{proof}

%% file: cocoa2013_thm_barrier2up.tex
\begin{theorem}\label{thm:barrier2up}
Let $G$ be a factorizable graph, and  $X$ be an odd-maximal barrier of $G$. 
Let $D_1,\ldots, D_k$ be the DM-components of $\hxg{X}{G}$. 
Let  $\what{V}_1,\ldots, \what{V}_k$ be the partition  of $X\cup D_X$  
such that for each $i = 1, \ldots, k$,    
$\what{D}_i := G[\what{V}_i]$ is the expansion of $D_i$. 
Then,  for each $i = 1,\ldots, k$, 
$S_i := X\cap \what{V}_i$ coincides with 
a member of $\pargpart{G}{H_i}$ for some $H_i \in \mathcal{G}(G)$, 
and $\what{V}_i$
coincides with $\vupstar{H_i}\setminus \vup{S_i}$. 
%
\end{theorem}

%% file: cocoa2013_proof_thm_barrier2up.tex
\begin{proof}
Note that such a partition of $X\cup D_X$ surely exists
by Proposition~\ref{prop:dm2partition}. 
Let $M$ be a perfect matching of $G$. Let $i\in \{1,\ldots, k\}$. 

\begin{cclaim}\label{claim:nosaturate}
There is no $M$-saturated path between any two vertices of $S_i$. 
\end{cclaim}
\begin{proof}
This is immediate from Proposition~\ref{prop:x2xd}.
\qed 
\end{proof}

\begin{cclaim}\label{claim:containedinelem}
$S_i$ is contained in the same factor-component of $G$, say $H_i$.
\end{cclaim}
\begin{proof}
Take $u, v\in S_i$ arbitrarily. 
Note first that there is no $M$-saturated path between $u$ and $v$,  
by Claim~\ref{claim:nosaturate}. 
Additionally, 
there are $M$-balanced paths from $u$ to $v$ and 
from $v$ to $u$ respectively,  
which is immediate from Theorem~\ref{thm:path2dm} and Proposition~\ref{prop:dm}. 
Therefore by Theorem~\ref{thm:char2elem}, 
$u$ and $v$ are contained in the same factor-component. 
Thus, we have the claim. 
\qed
\end{proof} 
Since $\what{V}_i$ is separating by Proposition~\ref{prop:dm2partition}, 
\begin{cclaim}\label{claim:contained}
$V(H_i) \subseteq \what{V}_i$. 
\end{cclaim}
%
\begin{cclaim}\label{claim:saturate}
For any $u\in S_i$ and any $v\in \what{V}_i \setminus S_i$, 
there is an $M$-saturated path between $u$ and $v$  
whose vertices are contained in $\what{V}_i$. 
\end{cclaim}
\begin{proof} 
Note that $M_{\what{V}_i}$ is a perfect matching of $\what{D}_i$, 
$S_i$ is an odd-maximal barrier of $\what{D}_i$,  and 
$\hxg{S_i}{\what{D}_i}$ is a factorizable bipartite graph with
exactly one DM-component by Proposition~\ref{prop:dm2partition}.
Thus, by applying Theorem~\ref{thm:path2dm} to $\what{D}_i$, $M_{\what{V}_i}$ and $S_i$,  
there is an $M$-saturated path  
between any $u\in S_i$ and any $v\in \what{V}_i\setminus S_i$,
which is contained in $\what{V}_i$.
%
\qed
\end{proof}
By combining Claims~\ref{claim:nosaturate}, 
\ref{claim:containedinelem}, \ref{claim:contained},  and \ref{claim:saturate}, 
we obtain that 
$S_i$ is a maximal subset of $V(H_i)$ such that 
there is no $M$-saturated path between any two vertices of it. 
Hence, by Proposition~\ref{prop:deletable2path}, 
$S_i\in \pargpart{G}{H_i}$ holds. 

\begin{cclaim}\label{claim:supset} 
$\what{V}_i \supseteq \vupstar{H_i}\setminus \vup{S_i}$. 
\end{cclaim}
\begin{proof} 
Take $y\in \vupstar{H_i}\setminus \vup{S_i}$ arbitrarily. 
If $y\in S_i$, then of course $y\in \what{V}_i$. 
Hence hereafter let $y\in \vupstar{H_i}\setminus \vupstar{S_i}$,  
and let $T\in \pargpart{G}{H_i}\setminus \{S_i\}$ be such that $y\in\vupstar{T}$. 

Let $u \in S_i$. 
There is an $M$-saturated path $P$ between  $u$ and $y$ 
by Lemma~\ref{lem:path2up}.
Hence, by Proposition~\ref{prop:x2xd}, 
$y\in D_X$. 
Therefore, 
there exists $j\in\{1,\ldots, k\}$ such that $y\in \what{V}_j$. 
%
By Theorem~\ref{thm:path2dm} and Proposition~\ref{prop:dm}, 
$D_i \dm D_j$. 

If $i \neq j$,   
then by Lemma~\ref{lem:traverse}, 
$P$ has some internal vertices which belong to $S_j$. 
However, 
by Proposition~\ref{prop:x2xd}, 
there is no $M$-saturated path 
between any two vertices respectively in $S_i$ and $S_j$, 
and of course $V(P)\cap S_j$ is disjoint from $S_i$.  
This contradicts Lemma~\ref{lem:path2up_restated}. 
Hence, we obtain $i = j$;  
accordingly,  $\vupstar{H_i}\setminus \vup{S_i}$ is contained in $\what{V}_i$. 
\qed
\end{proof}

\begin{cclaim}\label{claim:subset}
$\what{V}_i \subseteq \vupstar{H_i}\setminus \vup{S_i}$. 
\end{cclaim}
\begin{proof} 
%
Let $z\in \what{V}_i \setminus V(H_i)$. 
By Claim~\ref{claim:saturate}, 
there is an $M$-saturated path $P$
between $z$ and some vertex of $S_i$ which is contained in $\what{V}_i$.  
Trace $P$ from $z$ and let $w$ be the first vertex we encounter that is in $V(H_i)$.
Since $V(H_i)$ is separating, 
$zPw$ is an $M$-balanced path from $z$ to $w$ by Proposition~\ref{prop:separating}. 
In $\what{D}_i/H_i$, $zPw$ corresponds to  an $M$-balanced path from 
$z$ to the contracted vertex $h$, corresponding to $H_i$. 
Obviously, 
$M$ contains a near-perfect matching of $\what{D}_i/H_i$ 
exposing only $h$. 
%

Therefore, 
$\what{D}_i/H_i$ is factor-critical by Proposition~\ref{prop:path2root}; 
accordingly, $\what{V}_i$ is contained in $\vupstar{H_i}$.
Additionally, 
by Claim~\ref{claim:saturate} again and Lemma~\ref{lem:nosaturate}, 
we can see that $\what{V}_i$ is disjoint from $\vup{S_i}$ and that 
$\what{V}_i$ is contained in $\vupstar{H_i}\setminus \vup{S_i}$.  
\qed
\end{proof}
Thus, by Claims~\ref{claim:supset} and \ref{claim:subset}, 
we have $\what{V}_i = \vupstar{H_i}\setminus \vup{S_i}$. 
%
%
%
%
%
\qed  
\end{proof}

%% file: cocoa2013_remark_special.tex
\begin{remark}
If $G$ in Theorem~\ref{thm:barrier2up} is elementary, 
then $k = 1$ and $\what{V}_1 = V(G)$, 
which follows by Propositions~\ref{prop:barrier} and \ref{prop:dm2partition}. 
Therefore, in this case, 
Theorem~\ref{thm:barrier2up} claims that 
$\gpart{G}$ is the family of (odd-) maximal barriers;  
namely, 
Theorem~\ref{thm:barrier2up} coincides with Theorem~\ref{thm:canonicalpartition}. 
Therefore, 
Theorem~\ref{thm:barrier2up} can be regarded as a 
generalization of Theorem~\ref{thm:canonicalpartition}. 
\end{remark} 

%% file: cocoa2013_remark_atom.tex
\begin{remark}
Let $G$ be a factorizable graph. 
For an arbitrary vertex $x\in V(G)$, 
take a maximal barrier of $G-x$, say $X$. 
Then, $X\cup \{x\}$ is a maximal barrier of $G$; 
namely, for any vertex $x$ there is an odd-maximal barrier that contains $x$. 
Therefore, 
for any $S\in \gpart{G}$, 
there exists an odd-maximal barrier that contains $S$. 
\end{remark}  

%% file: cocoa2013_remark_nonfactorizable.tex
\begin{remark}
%
With Kir\'aly~\cite{kiraly1998}, if $G$ is a non-factorizable graph, 
then  $\{A(G)\} \cup \gpart{G[C(G)]}$ are the ``atoms'' that constitute odd-maximal barriers. 
For each odd-maximal barrier $X$, 
the odd components of $G-X$ are 
the components of $G[D(G)]$ and 
the odd components of $G[C(G)]-(X\setminus A(G))$; 
here $G[C(G)]$ forms a factorizable graph 
and $X\setminus A(G)$ is an odd-maximal barrier. 
%
%
\end{remark}

%% file: isaac2012_1_sec_algeff.tex
Hereafter we denote by $n$ and $m$ the 
number of vertices and edges (resp. arcs) of input graph (resp. digraph), 
 respectively. 
Note that factorizable graphs satisfy 
$m = {\rm \Omega}(n)$ and accordingly $O(n+m) = O(m)$.

In \cite{kita2012b, kita2012a}, 
we show that the partial order $\yield$ and 
the generalized canonical partition can be computed in 
$O(nm)$ time if there input a factorizable graph. 
The algorithm is composed of three stages, 
each of which is $O(n)$ times iteration of $O(m)$ time procedure of 
growing alternating trees.  
It first computes the factor-components, 
then computes
$\yield$ and $\gpart{G}$ respectively. 

With the results in this paper, 
we present another $O(nm)$ time algorithm to compute them. 
The upper bound of its time complexity is the same as the known one, 
however  
the factor-components, 
$\yield$, and $\gpart{G}$ 
are here computed simultaneously. 
Thus, it has some possibility of exhibiting a bit more efficiency. 

%
%
%
%
%
%
\begin{theorem}[Micali \& Vazirani~\cite{mv1980}, Vazirani~\cite{vazirani1994}]%
\label{thm:matching}
A maximum matching of a graph can be computed in 
$O(\sqrt{n}m)$ time. 
\end{theorem}
\begin{theorem}[Edmonds~\cite{edmonds1965}, Tarjan~\cite{tarjan1983}, %
 Gabow \& Tarjan~\cite{gt1985}]\label{thm:ge-alg}
Let $G$ be a graph with $m = {\rm \Omega}(n)$ 
and suppose we are given a perfect matching of $G$. 
Then, $D(G)$, $A(G)$, and $C(G)$ can be computed in 
$O(m)$ time. 
\end{theorem}
\begin{theorem}[Dulmage \& Mendelsohn~%
\cite{dm1958, dm1959, dm1963, murota2000}]\label{thm:dm-alg}
For any bipartite factorizable graph $G$, 
the Dulmage-Mendelsohn decomposition of $G$ 
can be computed in $O(m)$ time. 
\end{theorem}
%
%
%
%
\begin{proposition}[folklore, see~\cite{murota2000}]\label{prop:stronglyconnected-poset}
Let $D$ be a digraph, and $\mathcal{D}$ be the set of strongly-connected components of $D$. 
For $D_1, D_2\in \mathcal{D}$ 
we say $D_1\rightarrow D_2$ if 
for any $u\in V(D_1)$ and any $v\in V(D_2)$ 
there is a dipath from $u$ to $v$. 
Then, $\rightarrow$ is a partial order on $\mathcal{D}$. 
\end{proposition}
\begin{proposition}[folklore, see~\cite{schrijver2003}]\label{prop:stronglyconnected-alg}
For any digraph $D$, the strongly connected components of $D$ 
can be computed in $O(n + m)$ time. 
\end{proposition}
Below is the new algorithm, Algorithm~1:  

\begin{algorithmic}[1]
\REQUIRE a factorizable graph $G$
\ENSURE the generalized canonical partition $\gpart{G}$
and the digraph $Aux(G)$ representing $(\mathcal{G}(G), \yield)$ 
\STATE \label{line:matching} compute a perfect matching $M$ of $G$;
\STATE \label{line:initialize} 
$U := V(G)$; initialize $f:V(G) \rightarrow \{0, 1\}$ by $0$;
\STATE \label{line:initializea}$A := \emptyset$; $\gpart{G} := \emptyset$; 
\WHILE{$U\neq \emptyset$} \label{line:while} %
\STATE \label{line:chooseu}choose $u\in U$;
\STATE \label{line:alt} compute $X:= A(G-u)\cup \{u\}$; 
\STATE \label{line:dm} compute  the DM-decomposition of $\hxg{X}{G}$;
\FORALL{DM-component $D$ of $\hxg{X}{G}$} \label{line:forall_dm} 
\STATE \label{line:choosev}
let $S := X\cap V(D)$;
choose arbitrary $v\in S$;  
\IF{$f(v) = 0$}\label{line:testv} 
\STATE \label{line:part} 
$\gpart{G} := \gpart{G}\cup \{S\}$;
\STATE \label{line:hatd}let $\what{D}\subseteq G$ be the expansion of $D$; 
\FORALL{$x\in S$} \label{line:forallx}
\FORALL{$y\in V(\what{D})\setminus X$} \label{line:forally}
\STATE \label{line:updatea}$A := A\cup \{(x, y)\}$; 
\ENDFOR
\STATE \label{line:updateu}$U:= U\setminus \{x\}$; $f(x) := 1$;
\ENDFOR \label{line:endforall}
\ENDIF \label{line:endif} 
\ENDFOR \label{line:endfor_dm} 
\ENDWHILE \label{line:endwhile} %
\STATE \label{line:outputp} output $\gpart{G}$; 
\STATE \label{line:outputaux}$Aux(G):= (V(G), A )$;
decompose $Aux(G)$ into its strongly-connected components 
and output it;  STOP. 
\end{algorithmic}
%
\begin{proposition}\label{prop:alg}
While Algorithm~1 is running, 
\renewcommand{\labelenumi}{\theenumi}
\renewcommand{\labelenumi}{{\rm \theenumi}}
\renewcommand{\theenumi}{(\roman{enumi})}
\begin{enumerate}
\item \label{item:alt}
$X = A(G-u)\cup \{u\}$ of Line~\ref{line:alt} is an odd-maximal barrier of $G$, 
\item \label{item:choosev}
$S$ defined at Line~\ref{line:choosev} coincides with a member of $\gpart{G}$, and 
\item \label{item:forally}
$V(\what{D})\setminus X$ at Line~\ref{line:forally} coincides with $\vcoup{S}$%
\footnote{Given $H\in\mathcal{G}(G)$ and $S\in\pargpart{G}{H}$, 
we denote $\vupstar{H}\setminus \vupstar{S}$ as $\vcoup{S}$}.  
\end{enumerate}
\end{proposition}
\begin{proof}
The statement \ref{item:alt} follows by a simple counting argument. 
Therefore, \ref{item:choosev} and \ref{item:forally} 
follows by Theorem~\ref{thm:barrier2up}. 
\qed
\end{proof}
\begin{lemma}\label{lem:necessity}
Let $G$ be a factorizable graph and $Aux(G) =(V(G), A)$ be the digraph 
obtained by inputting $G$ to Algorithm~1.
Let $H_1, H_2\in\mathcal{G}(G)$, $u\in V(H_1)$, and $v\in V(H_2)$. 
\renewcommand{\labelenumi}{\theenumi}
\renewcommand{\labelenumi}{{\rm \theenumi}}
\renewcommand{\theenumi}{(\roman{enumi})}
\begin{enumerate}
\item \label{item:edge} If $(u, v)\in A$, then $H_1\yield H_2$.
\item \label{item:path} If there exists a dipath from $u$ to $v$ in $Aux(G)$, then $H_1\yield H_2$.
\end{enumerate}
\end{lemma}
\begin{proof}
The arc $(u, v)$ is added to $A$ only at Line~\ref{line:updatea} if 
$u\in X\cap V(\what{D})$ and $v\in V(\what{D})\setminus X$. 
Thus \ref{item:edge} follows by Proposition~\ref{prop:alg}.
Hence \ref{item:path} follows by the transitivity of $\yield$. 
\qed
\end{proof}
\begin{lemma}\label{lem:sufficiency}
Let $G$ be a factorizable graph and $Aux(G) =( V(G), A)$ be the digraph 
obtained by inputting $G$ to Algorithm~1.
Let $H_1, H_2\in\mathcal{G}(G)$ be such that $H_1\yield H_2$.
Then, for any $u\in V(H_1)$ and $v\in V(H_2)$,
there is a dipath from $u$ to $v$ in $Aux(G)$.
\end{lemma}
\begin{proof}
Let $S\in\pargpart{G}{H_1}$ be such that $u\in S$. 
First suppose that $v\in \vcoup{S}$. 
Then, $(u, v)$ is added to $A$ at Line~\ref{line:updatea} 
when $X\cap V(D)$ of Line~\ref{line:forallx} coincides with $S$, 
which surely occurs by Proposition~\ref{prop:alg}. 
Hence, the claim holds for this case. 
 
Now suppose the other case that  $v\in \vupstar{S}$. 
Take $T\in \pargpart{G}{H_1}\setminus \{S\}$ and $w\in T$ arbitrarily. 
The arc $(u, w)$ is added to $A$ at Line~\ref{line:updatea}
 when $S$ coincides with $X \cap V(D)$ of Line~\ref{line:forallx},
so is the arc $(w, v)$ when $T$ coincides with $X\cap V(D)$. 
Therefore the dipath $uw + wv$ satisfies the claim for this case, 
and we are done. 
\qed
\end{proof}
%
%
%
%
%
\begin{theorem}\label{thm:iff}
Let $G$ be a factorizable graph and $Aux(G) =( V(G), A)$ be the digraph 
obtained by inputting $G$ to Algorithm~1. 
Then, $H\in \mathcal{G}(G)$ holds if and only if 
there is a strongly-connected component $D$ of $Aux(G)$ 
with $V(H) = V(D)$. 
Additionally, 
for any $H_1, H_2\in \mathcal{G}(G)$, 
$H_1\yield H_2$ holds if and only if 
$D_1\rightarrow D_2$, 
where $D_i$ is the strongly-connected component of $Aux(G)$ 
with $V(H_i) = V(D_i)$, for each $i = 1, 2$. 
\end{theorem}
\begin{proof}
Combining Lemmas~\ref{lem:necessity} and \ref{lem:sufficiency}, 
we immediately obtain the following claim:  
\begin{cclaim}
$H_1 \yield H_2$ holds if and only if 
for any $u\in V(H_1)$ and any $v\in V(H_2)$ 
there is a dipath from $u$ to $v$ in $Aux(G)$. 
\end{cclaim}
Therefore, we are done by Proposition~\ref{prop:stronglyconnected-poset}. 
\qed
\end{proof}
\begin{theorem}~\label{thm:orderalgo}
Given a factorizable graph $G$, 
the poset $(\mathcal{G}(G), \yield)$ 
and the generalized canonical partition $\gpart{G}$
can be computed in $O(nm)$ time by Algorithm~1. 
\end{theorem}
\begin{proof}
The correctness  follows by Proposition~\ref{prop:alg} and Theorem~\ref{thm:iff}. 

Hereafter we prove the complexity.
Line~\ref{line:matching} costs $O(\sqrt{n}m)$ time by Theorem~\ref{thm:matching}. 
Line~\ref{line:initialize} costs $O(n)$ time,
 and Line~\ref{line:initializea} costs $O(1)$ time.
Lines~\ref{line:while} to \ref{line:dm} cost $O(m)$ time 
per each iteration of the while-loop in Line~\ref{line:while}. 
As the while-loop in Line~\ref{line:while} is repeated $O(n)$ times, 
they cost $O(nm)$ time  over the whole algorithm.

Each operations in Lines~\ref{line:forall_dm} to \ref{line:testv} 
costs $O(1)$ time per iteration, 
and they are iterated $O(n^2)$ time over the whole computation;  
therefore, they cost $O(n^2)$ time. 

Note that $f(v) = 0$ at Line~\ref{line:testv} holds true 
for at most $n$ times. 
Therefore, Lines~\ref{line:part} and \ref{line:hatd} cost $O(n)$ time. 
The number of repetition of Lines~\ref{line:forallx} to \ref{line:endif} 
is bounded by $|A| = O(n^2)$. 
Therefore, the operations there costs $O(n^2)$ over the algorithm.  
%
\qed
\end{proof}

%% file: cocoa2013_appendix_why.tex
Here we are going to explain more details on odd-maximal barriers 
which are omitted in Section~\ref{sec:aim}. 
Readers familiar with matching theory might skip this section. 
\subsection*{Maximal Barriers vs. Odd-maximal Barriers}
As we mention in Section~\ref{sec:aim}, 
for elementary graphs, the notion of maximal barriers 
and the notion of odd-maximal barriers are equivalent. 
This fact is easy to see using known properties; 
we are going to show it in the following.  
The next two propositions are to see Proposition~\ref{prop:elemoddmaximal}: 
%
\input{barriers_prop_proof_oddmaximal2maximal}
\input{barriers_prop_proof_elembarrier}

\input{barriers_prop_proof_elemoddmaximal}
Since maximal barriers are apparently odd-maximal barriers by the definitions, 
now we have that these two notions are equivalent for elementary graphs
 by Proposition~\ref{prop:elemoddmaximal}. 
\subsection*{Why It Suffices to Work on Factorizable Graphs} 
The following statements, leading to Proposition~\ref{prop:reduction2factorizable_nofoot}, 
show that 
in order to know canonical structures of odd-maximal barriers in general graphs, 
it suffices to work on factorizable graphs. 
%
\input{barriers_prop_increment}
\input{barriers_def_prop_ge}
Additionally, Kir\'aly shows that $A(G)$ 
is the minimum odd-maximal barriers in any graph $G$. 
%
\input{barriers_thm_intersection}
Therefore, combining up 
Proposition~\ref{prop:increment} and Proposition~\ref{prop:ge},
 and Theorem~\ref{thm:intersection}, 
we can see the following: 
\begin{proposition}
Let $G$ be a graph.
A set of vertices $S\subseteq V(G)$ is an odd-maximal barrier of $G$ 
if and only if it is a disjoint union of $A(G)$ and  
an odd-maximal barrier of the factorizable subgraph $G[C(G)]$. 
Now let $S$ be an odd-maximal barrier. 
Then,  the odd components of $G-S$ 
are the components of $G[D(G)]$ 
and the odd components of $G[C(G)]-(S\setminus A(G))$. 
\end{proposition}
Therefore, 
we can see that to obtain the structure of odd-maximal barriers
and the odd components associated with them in general graphs,  
it suffices to investigate factorizable graphs.

%% file: barriers_prop_proof_oddmaximal2maximal.tex
\begin{proposition}[see \cite{lp1986} or \cite{kiraly1998}]%
\label{prop:oddmaximal2maximal}
Let $G$ be a graph and  $X\subseteq V(G)$ be an odd-maximal barrier of $G$. 
Then, $X$ is a maximal barrier if and only if 
$C_X = \emptyset$. 
\end{proposition}
\begin{proof}
The necessity part is obvious by the definition. 
For the sufficiency part, 
let $C_X \neq \emptyset $ and take $u\in C_X$ arbitrarily. 
Then $X\cup \{u\}$ is also a barrier of $G$, 
contradicting $X$ being a maximal one. 
\qed
\end{proof}

%% file: barriers_prop_proof_elembarrier.tex
\begin{proposition}[see \cite{lp1986} or \cite{kiraly1998}]%
\label{prop:elembarrier}
Let $G$ be an elementary graph and $X$ be a barrier of $G$. 
Then, $C_X = \emptyset$. 
\end{proposition}
\begin{proof}
If $C_X \neq \emptyset$, then 
since no the edges of $E[X, C_X]$ are allowed as stated in Proposition~\ref{prop:barrier}, 
we can see that $G$ is not elementary, a contradiction. 
\qed
\end{proof}

%% file: barriers_prop_proof_elemoddmaximal.tex
\begin{proposition}\label{prop:elemoddmaximal}
For an elementary graph $G$, 
if $X\subseteq V(G)$ is an odd-maximal barrier  
then it is also a maximal barrier. 
\end{proposition}
\begin{proof}
This is by combining Proposition~\ref{prop:oddmaximal2maximal}
and Proposition~\ref{prop:elembarrier}. 
\qed
\end{proof}

%% file: barriers_prop_increment.tex
\begin{proposition}[folklore, see \cite{lp1986} or \cite{kiraly1998}]%
\label{prop:increment}
Let $G$ be a graph, 
$X\subseteq V(G)$ be a barrier of $G$, 
and $Y\subseteq V(G)$ be such that $X\subseteq Y$. 
Then, 
$Y$ is a barrier of $G$ 
if and only if $Y\setminus X$ is a union of barriers of 
some connected components of $G-X$. 
\end{proposition}

%% file: barriers_def_prop_ge.tex
Given a graph $G$, 
we define $D(G)$ as the set of vertices 
that can be respectively exposed by maximum matchings,
$A(G)$ as $\Gamma(D(G))$
and $C(G)$ as 
$V(G)\setminus (D(G)\cup A(G))$. 
There is a well-known  theorem stating that 
$A(G)$ forms a barrier with special properties, 
called the {\em Gallai-Edmonds structure theorem}~\cite{lp1986}; 
the next one is a part of it. 
\begin{proposition}\label{prop:ge}
Let $G$ be a graph. 
Then, $A(G)$ is an odd-maximal barrier of $G$ 
such that $D_{A(G)} = D(G)$ and $C_{A(G)} = C(G)$.  
\end{proposition}

%% file: barriers_thm_intersection.tex
%
\begin{theorem}[Kir\'ary~\cite{kiraly1998}]\label{thm:intersection}
Let $G$ be a graph, 
and $\mathcal{X}\subseteq 2^{V(G)}$ be the family of the odd-maximal barriers of $G$. 
Then, $\bigcap_{X\in \mathcal{X}}X = A(G)$. 
\end{theorem}